\documentclass{amsart}
\usepackage[latin9]{inputenc}
\usepackage{verbatim}
\usepackage{amsthm}
\usepackage{amssymb}

\makeatletter
\numberwithin{equation}{section}
\numberwithin{figure}{section}

\usepackage{url}

\numberwithin{equation}{section}

\theoremstyle{plain}
\newtheorem{theorem}{}[section]
\newtheorem{lemma}[theorem]{}

\theoremstyle{definition}
\newtheorem{definition}[theorem]{}

 \DeclareMathOperator{\M}{M}
 \DeclareMathOperator{\pn}{p_n}
 \DeclareMathOperator{\n}{n}
 \DeclareMathOperator{\recip}{r}
 \DeclareMathOperator{\ad}{{\mathbf{ad}}}
 \DeclareMathOperator{\phiop}{\phi}
 \DeclareMathOperator{\GCD}{GCD}

 \newcommand{\Q}{{\mathbb{Q}}}
 \newcommand{\Z}{{\mathbb{Z}}}

\makeatother

\begin{document}

\title[Roots with common tails]{Roots with common tails}

\maketitle

\author{{[}}

A.Hobby{]}{Alexandra Hobby}

\email{Alexandra.F.Hobby@gmail.com}


\author{{[}}

D. Hobby{]}{David Hobby}

\email{hobbyd@newpaltz.edu}

\address{Mathematics, SUNY New Paltz\\
1 Hawk Drive\\
New Paltz, NY, 12561\\
USA}


\thanks{This research was supported by a SURE grant from SUNY New Paltz.}


11A55; Secondary: 12F10.


\begin{abstract}
Some cubic polynomials over the integers have three distinct real
roots with continued fractions that all have the same common tail.
We characterize the polynomials for which this happens, and then investigate
the situation for other polynomials of low degree. 
\end{abstract}


\section{Introduction}

\label{S:intro}

Around 35 years ago, the second author used a computer to calculate
the roots of cubic polynomials and their continued fractions, when
he noticed an interesting phenomenon. Sometimes the three roots would
have continued fractions that all agreed after a certain point. For
example, the three roots of $x^{3}+6x^{2}+9x+1$ are approximately
$-3.5320888$, $-2.3472963$, and $-.1206147$, and the continued
fractions corresponding to these roots are $[-4;2,7,3,2,3,1,1,\dots]$,
$[-3,1,1,1,7,3,2,3,1,1,\dots]$, $[-1;1,7,3,2,3,1,1,\dots]$. It is
natural to say that the three roots { \em have common tails}. We
have found little prior mention of this phenomenon in the literature.

For background on continued fractions, the reader may turn to \cite{Olds}
or many other introductory texts. Also, \cite{Perron} is a very thorough
text, and contains our Theorem \ref{ad theorem}. (As do \cite{Hua}
and \cite{Hardy}.) For background on field extensions and Galois
theory, many advanced undergraduate texts in abstract algebra will
be fine. We will work with irreducible polynomials over $\Q$ or equivalently
over $\Z$, which have the form $ax^{3}+bx^{2}+cx+d$, where $a,b,c,d\in\Z$.
These polynomials will have distinct real roots arbitrarily called
$r_{1}$, $r_{2}$, and so on.

In general, the splitting field of an irreducible cubic polynomial
over the rationals has degree 6. But if the three roots have common
tails, adjoining any one root to $\Q$ also adds the common tail,
which in turn adds the other two roots. Thus the degree of the splitting
field must be 3. Using the discriminant $\Delta=a^{4}(r_{1}-r_{2})^{2}(r_{2}-r_{3})^{2}(r_{3}-r_{1})^{2}$,
we have that the splitting field has degree 3 if and only if $\Delta$
is a perfect square in $\Z$. In this case, the roots $r_{1}$, $r_{2}$
and $r_{3}$ are real and distinct. Using that $\Delta$ is $b^{2}c^{2}-4ac^{3}-4b^{3}d-27a^{2}d^{2}+18abcd$,
it is easy to go through various polynomials searching for those that
are irreducible and have splitting fields of degree 3.

The initial search yielded a small number of cubic polynomials with
splitting fields of degree 3, all of which had roots with common tails.
The natural conjecture was that the roots had common tails whenever
the splitting field had degree 3. Aside from mentioning the problem
to various number theorists, this is where the matter sat. Meanwhile,
computers and software had become much more powerful. So when we started
researching the topic in earnest in 2014, the first author used Mathematica
(\cite{Mathematica}) to check a large number of polynomials, and
promptly reported that the conjecture was false.

Given two irrational numbers $s$ and $t$, we view their continued
fractions as infinite sequences $[s_{0};s_{1},s_{2},s_{3},\dots]$
and $[t_{0};t_{1},t_{2},t_{3},\dots]$. Then $s$ and $t$ have {\em
common tails} if there exist $m$ and $n$ so that $s_{m+k}=t_{n+k}$
for all $k\geq0$. We write $s\approx t$ to mean than $s$ and $t$
have common tails, and note that $\approx$ is an equivalence relation.
In number theory, numbers with common tails are referred to as {\em
equivalent}. When there is no danger of confusion, we may also use
this term. (Rational numbers have terminating continued fractions,
and it would be natural to extend our definition by saying that all
rationals had common tails. While much of our theory would work in
this broader domain, we will always be working with irrational numbers.)

We restate some well-known facts in the following lemma. \begin{lemma}
\label{common tail lemma} If $r$ is irrational, then 
\begin{enumerate}
\item $r\approx r+n$ for any integer $n$, 
\item $r\approx-r$, and 
\item $r\approx1/r$ 
\end{enumerate}
\end{lemma} \begin{proof} Let $r$ have continued fraction $[r_{0};r_{1},r_{2},r_{3},\dots]$.
Then $r+n$ has continued fraction $[r_{0}+n;r_{1},r_{2},r_{3},\dots]$,
showing $r\approx r+n$.

The continued fraction of $-r$ is $[-r_{0}-1;1,r_{1}-1,r_{2},r_{3},\dots]$
if $r_{1}>1$, and is $[-r_{0}-1;r_{2}+1,r_{3},\dots]$ if $r_{1}=1$.
Either way, $r\approx-r$.

Since $r\approx-r$, we need only prove $r\approx1/r$ when $r$ is
positive, so assume that is the case. If $r_{0}=0$, then the continued
fraction of $1/r$ is $[r_{1};r_{2},r_{3},\dots]$. If $r_{0}>0$,
the continued fraction of $1/r$ is $[0;r_{0},r_{1},r_{2},r_{3},\dots]$.
\end{proof}


\section{Linear Fractional Transformations}

\label{S:Linear}

A {\em linear fractional transformation} is a map that takes $z$
to $(\alpha z+\beta)/(\gamma z+\delta)$. (These are also sometimes
called homographies, or Möbius transformations.) While these maps
are used in complex analysis (see \cite{Zill}) and other fields,
we will not need any outside results in this paper. Observe that the
composition of two linear fractional transformations is again a linear
fractional transformation. If $f(z)$ is $(\alpha z+\beta)/(\gamma z+\delta)$,
it is convenient to consider $f$ to be the class of all matrices
that correspond to choices of $\alpha$, $\beta$, $\gamma$ and $\delta$
that give the function $f$. Thus we define 
\[
\M(f)=\left\{ \lambda\left(\begin{array}{cc}
\alpha & \beta\\
\gamma & \delta
\end{array}\right):\lambda\neq0\right\} 
\]
We say that elements of $\M(f)$ are {\em matrices} of $f$, or
alternatively, {\em matrices} of $\alpha$, $\beta$, $\gamma$
and $\delta$. Where $v=f(u)$ and $u$ is understood, we will also
refer to a {\em matrix} of $v$.

We will mostly be concerned with linear fractional transformations
which are defined and not constant, these correspond to invertible
matrices. It is easily verified for the composition of the linear
fractional transformations $f$ and $g$, that $\M(f\circ g)$ is
the product $\M(f)\M(g)=\{AB:A\in\M(f),B\in\M(g)\}$.

Letting $ax^{3}+bx^{2}+cx+d$ be an irreducible polynomial over $\Z$
with splitting field of degree 3, we have that any element of $\Q(r_{1})$
can be uniquely written as $sr_{1}^{2}+tr_{1}+u$ for some $s,t,u\in\Q$.
To rewrite this element as a linear fractional transformation of $r_{1}$,
it is enough to deal with the case where $s\ne0$. We let $\gamma=a/s$,
and $\delta=b/s-at/s^{2}$. Then $(\gamma r_{1}+\delta)(sr_{1}^{2}+tr_{1}+u)=ar_{1}^{3}+br_{1}^{2}+(au/s+bt/s-at^{2}/s^{2})r_{1}+(b/s-at/s^{2})=\alpha r_{1}+\beta$,
where $\alpha=au/s+bt/s-at^{2}/s^{2}-c$ and $\beta=b/s-at/s^{2}-d$,
the last step since $ar_{1}^{3}+br_{1}^{2}+cr_{1}+d=0$. Thus $sr_{1}^{2}+tr_{1}+u=(\alpha r_{1}+\beta)/(\gamma r_{1}+\delta)$.
Multiplying top and bottom by a rational number, we can put any such
linear fractional transformation into a unique {\em standard form}
where $\alpha$, $\beta$, $\gamma$ and $\delta$ are integers that
do not all have a common factor and where either $\alpha$ is positive
or $\alpha$ is zero and $\beta$ is non-negative. We will also call
the matrix with entries these $\alpha$, $\beta$, $\gamma$ and $\delta$
the {\em standard matrix} of the linear fractional transformation.

Thus we may write $r_{2}=(\alpha r_{1}+\beta)/(\gamma r_{1}+\delta)$.
We let $\phi$ be the Galois automorphism of $\Q(r_{1})$ that fixes
$\Q$ and has $\phiop(r_{1})=r_{2}$, $\phiop(r_{2})=r_{3}$ and $\phiop(r_{3})=r_{1}$.
Applying $\phi$ repeatedly to $r_{2}=(\alpha r_{1}+\beta)/(\gamma r_{1}+\delta)$,
we obtain $r_{3}=(\alpha r_{2}+\beta)/(\gamma r_{2}+\delta)$ and
then $r_{1}=(\alpha r_{3}+\beta)/(\gamma r_{3}+\delta)$. This implies
that applying $f(z)=(\alpha z+\beta)/(\gamma z+\delta)$ three times
takes $r_{1}$ back to $r_{1}$. Thus the cube of any matrix in $\M(f)$
is a non-zero multiple of the identity matrix.

Note that since the ordering of the roots is arbitrary, that we may
just as well be dealing with $\phi^{-1}$ as with $\phi$. Doing so
gives us the inverse of the linear fractional transformation $f$,
which has matrices that are non-zero multiples of the inverse of the
matrix with entries $\alpha$, $\beta$, $\gamma$ and $\delta$.
Modulo this, the linear fractional transformation is uniquely determined
by our particular polynomial.


\section{Main Results}

\label{S:main results}

In view of Lemma \ref{common tail lemma}, we make the following definition.

\begin{definition}\label{basic operations} The {\em basic operations}
on real numbers are: 
\begin{enumerate}
\item ``plus $n$'', where $\pn(y)$ in $y+n$, 
\item ``negation'', where $\n(y)$ is $-y$, and 
\item ``reciprocal'', where $\recip(y)$ is $1/y$. 
\end{enumerate}
\end{definition}

Note that $n$ and $r$ are their own inverses, and that the inverse
of $p_{n}$ is $p_{-n}$. While $n$ turns out to be redundant, it
is convenient to include it as a basic operation.

\begin{theorem}\label{chain of basic ops theorem} For any irrational
numbers $s$ and $u$, $s\approx u$ iff $u$ can be obtained from
$s$ by a composition of basic operations. \end{theorem} \begin{proof}
Let $s$ and $u$ be given. If $u$ can be obtained from $s$ by basic
operations, Lemma \ref{common tail lemma} implies $s\approx u$.
So assume $s\approx u$. We have that $s$ and $u$ have continued
fractions $[s_{0};s_{1},s_{2},\dots]$ and $[u_{0};u_{1},u_{2},\dots]$
where for some $m$ and $n$, $s_{m}=u_{n}$, $s_{m+1}=u_{n+1}$,
and so on. Let $t$ be the number represented by this common tail,
so $t$ has continued fraction $[s_{m};s_{m+1},s_{m+2},\dots]$. Then
$t$ is equal to $r(p_{-s_{m-1}}(\dots r(p_{-s_{1}}(r(p_{-s_{0}}(s))))\dots))$.
Similarly we have $t=r(-p_{u_{n-1}}(\dots r(p_{-u_{1}}(r(p_{-u_{0}}(u))))\dots))$,
so $u=(p_{u_{0}}\circ\dots p_{u_{n-1}}\circ r)\circ(r\circ p_{-s_{m-1}}\circ\dots p_{-s_{0}})(s)$.
\end{proof}

We define the operation $\ad$ on $2\times2$ matrices by letting
$\ad$ be the absolute value of the determinant, and also write 
\[
\ad(\alpha,\beta,\gamma,\delta)=\ad\left(\begin{array}{cc}
\alpha & \beta\\
\gamma & \delta
\end{array}\right)
\]

The key fact is that basic operations do not change the value of $\ad(\alpha,\beta,\gamma,\delta)$.
That is, suppose that $v=(\alpha u+\beta)/(\gamma u+\delta)$, and
let $M$ be the matrix of $v$. Then $\pn(v)=((\alpha+n\gamma)u+(\beta+n\delta))/(\gamma u+\delta)$,
and a matrix for $\pn(v)$ is obtained from $M$ by adding $n$ times
the bottom row of $M$ to the top row of $M$. The new matrix has
the same determinant as $M$ does. Similarly, $r$ corresponds to
interchanging the rows of $M$, and $n$ multiplies a row of $M$
by $-1$. Neither of these change the absolute value of the determinant.
(We could also have represented $p_{n}$, $r$ and $n$ as linear
fractional transformations, and noted that they had matrices with
determinants of $1$, $-1$ and $-1$, respectively.)

\begin{theorem}\label{ad theorem} Let $s$ and $t$ be irrational.
Then $s\approx t$ iff there are integers $\alpha$, $\beta$, $\gamma$
and $\delta$ where $t=(\alpha s+\beta)/(\gamma s+\delta)$ and $\ad(\alpha,\beta,\gamma,\delta)=1$.
\end{theorem} \begin{proof} Suppose that $s\approx t$. By Theorem
\ref{chain of basic ops theorem}, there is a sequence $s=u_{0},u_{1},u_{2},\dots u_{n}=t$,
where for $0\leq i<n$, $u_{i+1}$ is obtained by performing a basic
operation to $u_{i}$. We may write $u_{0}$ as $(1s+0)/(0s+1)$,
so a matrix of $u_{0}$ in terms of $s$ is the identity matrix. Since
basic operations do not change the absolute value of the determinant,
there are integers $\alpha$, $\beta$, $\gamma$ and $\delta$ where
$t=u_{n}=(\alpha s+\beta)/(\gamma s+\delta)$ and $\ad(\alpha,\beta,\gamma,\delta)=1$.

Now suppose that $t=(\alpha s+\beta)/(\gamma s+\delta)$ and $\ad(\alpha,\beta,\gamma,\delta)=1$.
We will row-reduce the matrix $M$ with entries $\alpha$, $\beta$,
$\gamma$ and $\delta$, using steps corresponding to the basic operations.
Note that $\GCD(\alpha,\gamma)$ divides $\ad(\alpha,\beta,\gamma,\delta)$,
so $\GCD(\alpha,\gamma)=1$. We can perform the Euclidean Algorithm
on the left column of $M$, reducing $M$ to a matrix $N$ with entries
$1$, $\beta'$, $0$ and $\delta'$. Since the absolute value of
the determinant of $N$ is $1$, $\delta'$ must be $1$ or $-1$.
If it is $-1$, apply the operation $n$ to negate the bottom row
of $N$. And then we apply $p_{-\beta'}$ to subtract $\beta'$ times
the bottom row from the top row, giving the identity matrix, which
is a matrix for $s$. This process gives us a chain of basic operations
that converts $t$ to $s$. By Theorem \ref{chain of basic ops theorem},
$s\approx t$. \end{proof}

Here is an example to illustrate the second half of the above proof.
Let $t$ be $(1s+3)/(2s+5)$ where $\ad(1,3,2,5)=|-1|=1$. Then the
row-reduction would be 
\[
%
\left(\begin{array}{cc}
1 & 3\\
2 & 5
\end{array}\right)\stackrel{\longrightarrow}{r}\left(\begin{array}{cc}
2 & 5\\
1 & 3
\end{array}\right)\stackrel{\longrightarrow}{p_{-2}}\left(\begin{array}{cc}
0 & -1\\
1 & 3
\end{array}\right)\stackrel{\longrightarrow}{r}\left(\begin{array}{cc}
1 & 3\\
0 & -1
\end{array}\right)\stackrel{\longrightarrow}{n}\left(\begin{array}{cc}
1 & 3\\
0 & 1
\end{array}\right)\stackrel{\longrightarrow}{p_{-3}}\left(\begin{array}{cc}
1 & 0\\
0 & 1
\end{array}\right)
\]

The above theorem has been known for a long time, and may have started
out as ``folklore''. It appears in \cite{Serret1} and \cite{Serret2}
by J. A. Serret, and is used by Hurwitz in \cite{Hurwitz} which is
on continued fractions with a generalized arithmetic pattern.

\begin{theorem}\label{characterization theorem} Given a cubic polynomial
over $\Q$ with a splitting field of degree $3$, its three roots
have common tails iff it has roots $r_{1}$ and $r_{2}$ where $r_{2}=(\alpha r_{1}+\beta)/(\gamma r_{1}+\delta)$
and $\ad(\alpha,\beta,\gamma,\delta)=1$. \end{theorem}

\begin{proof} Let $p(x)$ be a polynomial over $\Q$ with splitting
field of degree $3$. If $p(x)$ factored over $\Q$, its splitting
field would have degree $1$ or $2$, so $p(x)$ is irreducible, and
thus has $3$ distinct roots. If $p(x)$ has complex roots, it must
have a pair of them and one real root $r$. But then $\Q(r)$ has
degree $3$ and does not contain all roots of $p(x)$, a contradiction.
So $p(x)$ has three distinct real roots. If say $r_{2}=(\alpha r_{1}+\beta)/(\gamma r_{1}+\delta)$
where $\ad(\alpha,\beta,\gamma,\delta)=1$, then $r_{1}\approx r_{2}$
by the previous theorem. Applying the Galois automorphism with $\phiop(r_{1})=r_{2}$,
we get $r_{3}=\phiop(r_{2})=(\alpha\phiop(r_{1})+\beta)/(\gamma\phiop(r_{1})+\delta)=(\alpha r_{2}+\beta)/(\gamma r_{2}+\delta)$
where $\ad(\alpha,\beta,\gamma,\delta)=1$, so $r_{2}\approx r_{3}$
as well. \end{proof}

Consider our initial example of the polynomial $x^{3}+6x^{2}+9x+1$.
We found its roots with Mathematica, and used the command ``FindIntegerNullVector''
to produce integers $\alpha$, $\beta$, $\gamma$ and $\delta$ so
that $r_{2}=(\alpha r_{1}+\beta)/(\gamma r_{1}+\delta)$. (This command
uses the PSLQ integer relation algorithm. See \cite{Bailey} for some
interesting examples of what this algorithm can accomplish.) This
gave $\alpha=3$, $\beta=7$, $\gamma=-1$ and $\delta=-2$, which
we chose to have no common factor. Since $\ad(3,7,-1,-2)=1$, the
roots of $x^{3}+6x^{2}+9x+1$ have common tails. (If Mathematica had
numbered the roots differently we may have had $\alpha=-2$, $\beta=-7$,
$\gamma=1$ and $\delta=3$, corresponding to the inverse of the matrix
with entries $3$,$7$,$-1$ and $-2$. Since $\ad(-2,-7,1,3)$ is
also $1$, this makes no difference.)

Extending our methods slightly, let $r$ be some real root of an irreducible
cubic polynomial $p(x)$. Then every irrational element of $\Q(r)$,
can be written as $(\alpha r+\beta)/(\gamma r+\delta)$ in standard
form. We can characterize when two such elements $s=(\alpha r+\beta)/(\gamma r+\delta)$
and $t=(\alpha'r+\beta')/(\gamma'r+\delta')$ have common tails. Although
it is necessary that $\ad(\alpha,\beta,\gamma,\delta)=\ad(\alpha',\beta',\gamma',\delta')$,
it is not sufficient. Let $\epsilon$ be the GCD of $\alpha$ and
$\gamma$, and let $\eta$ be $\ad(\alpha,\beta,\gamma,\delta)/\epsilon$.
There are also congruence conditions modulo $\eta$.

Since the GCD of $\alpha/\epsilon$ and $\gamma/\epsilon$ is $1$,
there are integers $p$ and $q$ with $p(\alpha/\epsilon)+q(\gamma/\epsilon)=1$.
Let $y$ be such that $0\leq y<\eta$ and $y$ is congruent to $p\beta+q\delta$
mod $\eta$. Then working modulo $\eta$, we have $y(\alpha/\epsilon)\equiv(\alpha/\epsilon)(p\beta+q\delta)\equiv(\alpha/\epsilon)p\beta+(\alpha/\epsilon)q\delta+(\gamma/\epsilon)q\beta-(\gamma/\epsilon)q\beta\equiv(p(\alpha/\epsilon)+q(\gamma/\epsilon))\beta+q((\alpha/\epsilon)\delta-(\gamma/\epsilon)\beta)\equiv1\beta+q\eta\equiv\beta$.
Similarly, $y(\gamma/\epsilon)\equiv\delta$.

Applying any of the three basic operations to $(\alpha r+\beta)/(\gamma r+\delta)$
do not change $\epsilon$, which is the GCD of $\alpha$ and $\gamma$,
and $\eta$ = $\ad(\alpha,\beta,\gamma,\delta)/\epsilon$ is also
unchanged. The number $y$ is such that $y(\alpha/\epsilon)\equiv\beta$
and $y(\gamma/\epsilon)\equiv\delta$ modulo $\eta$. It is clear
that these congruences still hold after $n$ and $r$ are applied,
and applying $p_{n}$ gives $\alpha'=\alpha+n\gamma$ and $\beta'=\beta+n\delta$.
Then $y(\alpha'/\epsilon)\equiv y((\alpha/\epsilon)+n(\gamma/\epsilon))\equiv y(\alpha/\epsilon)+ny(\gamma/\epsilon))\equiv\beta+n\delta\equiv\beta'$,
showing that the congruences still hold. Thus basic operations preserve
the values of $\ad(\alpha,\beta,\gamma,\delta)$ and $\GCD(\alpha,\gamma)$,
as well as the fact that the two congruences involving $y$ hold.

Now consider an arbitrary $2\times2$ matrix with entries $\alpha$,
$\beta$, $\gamma$ and $\delta$, and define $\epsilon$, $\eta$
and $y$ as above. As in the proof of Theorem \ref{ad theorem}, doing
basic operations to apply the Euclidean Algorithm reduces the left
column to $\epsilon$ and $0$. Since $\ad(\alpha,\beta,\gamma,\delta)$
is preserved, the lower right entry is $\pm\eta$. Using $n$ if need
be, we make that entry $\eta$. Now we can apply $p_{n}$ with the
proper choice of $n$ so that the upper right entry $z$ satisfies
$0\leq z<\eta$, and none of the other entries changes. Since $y(\epsilon/\epsilon)\equiv z$
mod $\eta$, we have $z=y$. Note that this also implies that $y$
is unique. 
\[
\mbox{The reduction goes }\left(\begin{array}{cc}
\alpha & \beta\\
\gamma & \delta
\end{array}\right)\longrightarrow\left(\begin{array}{cc}
\epsilon & ?\\
0 & ?
\end{array}\right)\longrightarrow\left(\begin{array}{cc}
\epsilon & ?\\
0 & \eta
\end{array}\right)\longrightarrow\left(\begin{array}{cc}
\epsilon & y\\
0 & \eta
\end{array}\right)
\]

\begin{theorem}\label{same tail criterion} Let $r$ be a real root
of an irreducible cubic, and let $s=(\alpha r+\beta)/(\gamma r+\delta)$
and $t=(\alpha'r+\beta')/(\gamma'r+\delta')$ be any two irrational
elements of $\Q(r)$. Then $s$ and $t$ have common tails iff $\ad(\alpha,\beta,\gamma,\delta)=\ad(\alpha',\beta',\gamma',\delta')$,
$\GCD(\alpha,\gamma)=\GCD(\alpha',\gamma')$ and $s$ and $t$ have
the same value of $y$, where $y$ is computed as above. \end{theorem}

\begin{proof} Suppose $s$ and $t$ have common tails. Then Theorem
\ref{chain of basic ops theorem} implies that a sequence of basic
operations takes $s$ to $t$. Since basic operations do not change
the absolute value of the determinant, $\epsilon$, or $y$, these
have the same values for both $s$ and $t$.

Now assume $\ad(\alpha,\beta,\gamma,\delta)=\ad(\alpha',\beta',\gamma',\delta')$,
$\GCD(\alpha,\gamma)=\GCD(\alpha',\gamma')=\epsilon$ and that $s$
and $t$ have the same value of $y$. Letting $u=(\epsilon r+y)/(0r+\eta)$,
we have $s\approx u\approx t$ by Theorem \ref{chain of basic ops theorem}.
\end{proof}

Letting $u=(\epsilon r+y)/(0r+\eta)=(\epsilon/\eta)r+(y/\eta)$ as
in the proof, we observe that $\epsilon/\eta$ can be any positive
rational, and that $y/\eta$ can also be any rational in the interval
$[0,1)$. As a corollary, we have a set of representatives of the
$\approx$ equivalence classes, the set $\{\mu r+\nu\colon\mu,\nu\in\Q,0<\mu,0\leq\nu<1\}$.

We have a nice criterion for when roots of a cubic have common tails,
but it is not much use computationally. It would be better to be able
to tell if the roots have common tails without first having to find
the roots. Our first approach is to use the fact that a sequence of
basic operations permutes the roots of the cubic. For example, consider
the case where $\alpha=1$, $\beta=-1$, $\gamma=1$ and $\delta=0$.
We have 
\[
\left(\begin{array}{cc}
1 & -1\\
1 & 0
\end{array}\right)^{3}=\left(\begin{array}{cc}
-1 & 0\\
0 & -1
\end{array}\right)\mbox{and}\left(\begin{array}{cc}
1 & -1\\
1 & 0
\end{array}\right)=\left(\begin{array}{cc}
-1 & 0\\
0 & 1
\end{array}\right)\left(\begin{array}{cc}
1 & -1\\
0 & 1
\end{array}\right)\left(\begin{array}{cc}
0 & 1\\
1 & 0
\end{array}\right)
\]

The first equation shows that this choice of $\alpha$, $\beta$,
$\gamma$ and $\delta$ may be one that actually occurs, since doing
the associated fractional linear transformation three times would
take $r_{1}$ to $r_{2}$, on to $r_{3}$, and finally back to $r_{1}$,
since the cube of the matrix is a multiple of the identity. The second
equation factors the matrix into elementary matrices corresponding
to the basic operations, with that for $r$ on the right. This shows
that the linear fractional transformation is $n\circ p_{-1}\circ r$.

To find polynomials with this $\alpha$, $\beta$, $\gamma$ and $\delta$,
we consider the effect of $n\circ p_{-1}\circ r$ on the roots. If
$r_{1}$ is a non-zero root of $ax^{3}+bx^{2}+cx+d$, then $ar_{1}^{3}+br_{1}^{2}+cr_{1}+d=0$,
so $a+b(1/r_{1})+c(1/r_{1})^{2}+d(1/r_{1})^{3}$, and $1/r_{1}$ is
a root of $dx^{3}+cx^{2}+bx+a$. That is, reversing the order of the
coefficients gives a polynomial with roots the reciprocals of those
for the original polynomial. It is convenient to identify polynomials
with row vectors of their coefficients, so we have $R\langle a,b,c,d\rangle^{T}=\langle d,c,b,a\rangle^{T}$,
where $R$ is the matrix so that multiplying by it gives the polynomial
with roots reciprocal to the original roots. ($R$ has entries $R_{1,4}=R_{2,3}=R_{3,2}=R_{4,1}=1$,
and the rest of its entries are $0$.) Next note that when $r_{1}$
is a root of $p(x)$, $r_{1}-1$ is a root of $p(x+1)$. Applying
this to $ax^{3}+bx^{2}+cx+d$, we get the polynomial $a(x+1)^{3}+b(x+1)^{2}+c(x+1)+d=a(x^{3}+3x^{2}+3x+1)+b(x^{2}+2x+1)+c(x+1)+d=ax^{3}+(3a+b)x^{2}+(3a+2b+c)x+(a+b+c+d)$.
Representing polynomials as row vectors, we have the matrix $P_{-1}$
where $P_{-1}\langle a,b,c,d\rangle^{T}=\langle a,3a+b,3a+2b+c,a+b+c+d\rangle^{T}$.
Finally, we have that $r_{1}$ is a root of $p(x)$ iff $-r_{1}$
is a root of $p(-x)$. So the basic operation $n$ corresponds to
taking $ax^{3}+bx^{2}+cx+d$ to $a(-x)^{3}+b(-x)^{2}+c(-x)+d=-ax^{3}+bx^{2}-cx+d$,
or equivalently, to $ax^{3}-bx^{2}+cx-d$. This gives us a matrix
$N$ with $N\langle a,b,c,d\rangle^{T}=\langle a,-b,c,-d\rangle^{T}$.

If $p(x)=ax^{3}+bx^{2}+cx+d$ has $\alpha=1$, $\beta=-1$, $\gamma=1$
and $\delta=0$, then applying $NP_{-1}R$ to $\langle a,b,c,d\rangle^{T}$
should give us a vector corresponding to a polynomial with the same
roots as $p(x)$. Since minimal polynomials are unique to within a
constant factor, we have that $\langle a,b,c,d\rangle^{T}$ is an
eigenvector of $NP_{-1}R$. Calculation shows that 
\[
NP_{-1}R=\left(\begin{array}{cccc}
1 & 0 & 0 & 0\\
0 & -1 & 0 & 0\\
0 & 0 & 1 & 0\\
0 & 0 & 0 & -1
\end{array}\right)\left(\begin{array}{cccc}
1 & 0 & 0 & 0\\
3 & 1 & 0 & 0\\
3 & 2 & 1 & 0\\
1 & 1 & 1 & 1
\end{array}\right)\left(\begin{array}{cccc}
0 & 0 & 0 & 1\\
0 & 0 & 1 & 0\\
0 & 1 & 0 & 0\\
1 & 0 & 0 & 0
\end{array}\right)=\left(\begin{array}{cccc}
0 & 0 & 0 & 1\\
0 & 0 & -1 & -3\\
0 & 1 & 2 & 3\\
-1 & -1 & -1 & -1
\end{array}\right)
\]
This matrix has two complex eigenvalues, and a repeated real eigenvalue
of $1$ with eigenspace spanned by $\langle1,-3,0,1\rangle^{T}$ and
$\langle0,-1,1,0\rangle^{T}$. We rewrite the linear combinations
$a\langle1,-3,0,1\rangle+c\langle0,-1,1,0\rangle$ as polynomials,
and get that all irreducible cubics of the form $ax^{3}+(-3a-c)x^{2}+cx+a$
have $\alpha=1$, $\beta=-1$, $\gamma=1$ and $\delta=0$.

A method which works well in practice for determining if one of our
cubic polynomials has roots with common tails can now be loosely stated
as follows. ``Keep transforming the polynomial using $n$, $r$ and
$p_{k}$ for appropriate values of $k$, trying to put it in the form
$ax^{3}+(-3a-c)x^{2}+cx+a$. If you succeed, the original polynomial
has roots with common tails.'' The idea behind this method is that
the transformations $n$, $r$ and $p_{k}$ do not change whether
the roots of a polynomial have common tails, and that the transformations
can usually be strung together in a fashion reminiscent of the Euclidean
Algorithm to produce a monic polynomial with all its coefficients
small in absolute value.

Another perspective on using transformations to simplify polynomials
can be found by looking at what transformations do to the linear fractional
transformation that relates the roots. Suppose we have a cubic polynomial
with splitting field of degree $3$, where $r_{2}=(\alpha r_{1}+\beta)/(\gamma r_{1}+\delta)$
and $\ad(\alpha,\beta,\gamma,\delta)=1$. Then we must have 
\[
\left(\begin{array}{cc}
\alpha & \beta\\
\gamma & \delta
\end{array}\right)^{3}=\pm I\mbox{ , so }\pm\left(\begin{array}{cc}
\delta & -\beta\\
-\gamma & \alpha
\end{array}\right)=\left(\begin{array}{cc}
\alpha & \beta\\
\gamma & \delta
\end{array}\right)^{-1}=\pm\left(\begin{array}{cc}
\alpha & \beta\\
\gamma & \delta
\end{array}\right)^{2}\mbox{ .}
\]
Where we have the positive signs if $\alpha\delta-\beta\gamma=1$,
and the negative signs if $\alpha\delta-\beta\gamma=-1$. Either case
gives the same set of equations $\alpha^{2}+\beta\gamma=\delta$,
$\beta(\alpha+\delta)=-\beta$, $\gamma(\alpha+\delta)=-\gamma$,
and $\delta^{2}+\beta\gamma=\alpha$. If $\alpha+\delta\neq-1$, we
have $\beta=\gamma=0$ which makes $r_{2}=\pm r_{1}$. Thus $\delta=-1-\alpha$,
and both our remaining equations reduce to $\beta\gamma=-(1+\alpha+\alpha^{2})$.

This forces $|\alpha|$ and $|\delta|$ to be almost the same size,
as well as making $|\beta||\gamma|$ approximately the same size as
$|\alpha|^{2}$. Thus reducing the absolute value of one of $\alpha$,
$\beta$, $\gamma$ or $\delta$ essentially reduces the absolute
values of all the others. Applying $n$, $r$ or $p_{k}$ to a polynomial
has the effect of conjugating the linear fractional transformation
relating its roots by that corresponding to the corresponding basic
operation. If $f$ is a linear fractional transformation with $f(x)=(\alpha x+\beta)/(\gamma x+\delta)$,
we have that $n\circ f\circ n^{-1}(x)=-(\alpha(-x)+\beta)/(\gamma(-x)+\delta)=(\alpha x-\beta)/(-\gamma x+\delta)$.
That is, conjugation by $n$ negates $\beta$ and $\gamma$. Similar
calculations show that conjugation by $n$, $r$ and $p_{k}$ takes
\[
\left(\begin{array}{cc}
\alpha & \beta\\
\gamma & \delta
\end{array}\right)\mbox{to}\left(\begin{array}{cc}
\alpha & -\beta\\
-\gamma & \delta
\end{array}\right)\mbox{,}\left(\begin{array}{cc}
\delta & \gamma\\
\beta & \alpha
\end{array}\right)\mbox{and}\left(\begin{array}{cc}
(\alpha+k\gamma) & (\beta-k\alpha+k\delta-k^{2}\gamma)\\
\gamma & (\delta-k\gamma)
\end{array}\right)\mbox{,}
\]
respectively. Thus under our assumptions, one can usually simplify
the linear fractional transformation $(\alpha x+\beta)/(\gamma x+\delta)$
as follows. First, conjugate by $r$ if $|\gamma|>|\beta|$. Then
pick $k$ so that $|\alpha+k\gamma|$ is as small as possible, and
conjugate by $p_{k}$. Now repeat these steps until all absolute values
are as small as can be obtained. This will likely produce the new
values $\alpha'=1$, $\beta'=-1$, $\gamma'=1$ and $\delta'=0$,
possibly after conjugating by $r$ or $n$ as needed.

We can also use symmetric functions of the roots to write $\alpha$,
$\beta$, $\gamma$ and $\delta$ in terms of the coefficients of
the polynomial. Dividing through by the coefficient of $x^{3}$, we
may assume our polynomial is $x^{3}+bx^{2}+cx+d$, where $b$, $c$
and $d$ are rational. We can also write this polynomial in terms
of its roots $r_{1}$, $r_{2}$ and $r_{3}$ as $(x-r_{1})(x-r_{2})(x-r_{3})$
and obtain $r_{1}+r_{2}+r_{3}=-b$, $r_{1}r_{2}+r_{2}r_{3}+r_{3}r_{1}=c$
and $r_{1}r_{2}r_{3}=-d$.

Assume we have $r_{2}=(\alpha r_{1}+\beta)/(\gamma r_{1}+\delta)$,
or equivalently $\gamma r_{1}r_{2}+\delta r_{2}=\alpha r_{1}+\beta$.
Applying the Galois automorphism, we also have the two cyclicly permuted
equations $\gamma r_{2}r_{3}+\delta r_{3}=\alpha r_{2}+\beta$ and
$\gamma r_{3}r_{1}+\delta r_{1}=\alpha r_{3}+\beta$. Adding the three
equations gives $\gamma(r_{1}r_{2}+r_{2}r_{3}+r_{3}r_{1})+\delta(r_{1}+r_{2}+r_{3})=\alpha(r_{1}+r_{2}+r_{3})+3\beta$,
or $b\alpha-3\beta+c\gamma-b\delta=0$.

Next we take $\gamma r_{1}r_{2}+\delta r_{2}=\alpha r_{1}+\beta$,
and multiply it by $r_{3}$ to get $\gamma r_{1}r_{2}r_{3}+\delta r_{2}r_{3}=\alpha r_{1}r_{3}+\beta r_{3}$.
As before, the two cyclic permutations of this equation are also valid.
Adding all three together gives us $\gamma3(r_{1}r_{2}r_{3})+\delta(r_{1}r_{2}+r_{2}r_{3}+r_{3}r_{1})=\alpha(r_{1}r_{2}+r_{2}r_{3}+r_{3}r_{1})+\beta(r_{1}+r_{2}+r_{3})$
or $-c\alpha+b\beta-3d\gamma+c\delta=0$.

Strictly speaking, $\sqrt{(}\Delta)$ is $\pm(r_{1}-r_{2})(r_{2}-r_{3})(r_{3}-r_{1})$.
We may assume that the roots are ordered so that $\sqrt{(}\Delta)$
is $(r_{1}-r_{2})(r_{2}-r_{3})(r_{3}-r_{1})$, and will do so from
now on. We have $(r_{1}-r_{2})(r_{2}-r_{3})(r_{3}-r_{1})=(r_{1}r_{2}^{2}+r_{2}r_{3}^{2}+r_{3}r_{1}^{2})-(r_{1}^{2}r_{2}+r_{2}^{2}r_{3}+r_{3}^{2}r_{1})$,
and will let $\mu$ be $(r_{1}r_{2}^{2}+r_{2}r_{3}^{2}+r_{3}r_{1}^{2})$
and $\nu$ be $(r_{1}^{2}r_{2}+r_{2}^{2}r_{3}+r_{3}^{2}r_{1})$, so
$\sqrt{(}\Delta)=\mu-\nu$. Now we take our equation $\gamma r_{1}r_{2}+\delta r_{2}=\alpha r_{1}+\beta$,
and multiply it by $r_{3}^{2}$ to get $\gamma r_{1}r_{2}r_{3}^{2}+\delta r_{2}r_{3}^{2}=\alpha r_{1}r_{3}^{2}+\beta r_{3}^{2}$.
We also have the two cyclic permutations of this equation, $\gamma r_{2}r_{3}r_{1}^{2}+\delta r_{3}r_{1}^{2}=\alpha r_{2}r_{1}^{2}+\beta r_{1}^{2}$
and $\gamma r_{3}r_{1}r_{2}^{2}+\delta r_{1}r_{2}^{2}=\alpha r_{3}r_{2}^{2}+\beta r_{2}^{2}$.
Adding all three together gives us $\gamma(r_{1}r_{2}r_{3})(r_{1}+r_{2}+r_{3})+\delta(r_{1}r_{2}^{2}+r_{2}r_{3}^{2}+r_{3}r_{1}^{2})=\alpha(r_{1}^{2}r_{2}+r_{2}^{2}r_{3}+r_{3}^{2}r_{1})+\beta(r_{1}^{2}+r_{2}^{2}+r_{3}^{2})$
or $(-d)(-b)\gamma+\mu\delta=\nu\alpha+(r_{1}^{2}+r_{2}^{2}+r_{3}^{2})\beta$.

To simplify $r_{1}^{2}+r_{2}^{2}+r_{3}^{2}$, we calculate $b^{2}=(r_{1}+r_{2}+r_{3})^{2}=(r_{1}^{2}+r_{2}^{2}+r_{3}^{2})+2(r_{1}r_{2}+r_{2}r_{3}+r_{3}r_{1})=(r_{1}^{2}+r_{2}^{2}+r_{3}^{2})+2c$,
which gives $r_{1}^{2}+r_{2}^{2}+r_{3}^{2}=b^{2}-2c$. Substituting
this in our previous equation, we obtain $-\nu\alpha+(2c-b^{2})\beta+bd\gamma+\mu\delta=0$.

It remains to express $\mu$ and $\nu$ in terms of $b$, $c$ and
$d$. We have $\mu+\nu=(r_{1}r_{2}^{2}+r_{2}r_{3}^{2}+r_{3}r_{1}^{2})+(r_{1}^{2}r_{2}+r_{2}^{2}r_{3}+r_{3}^{2}r_{1})=(r_{1}r_{2}+r_{2}r_{3}+r_{3}r_{1})(r_{1}+r_{2}+r_{3})-3r_{1}r_{2}r_{3}=c(-b)-3(-d)=3d-bc$.
This gives us $\mu=1/2[(\mu+\nu)+(\mu-\nu)]=1/2[3d-bc+\sqrt{(}\Delta)]$
and $\nu=1/2[(\mu+\nu)-(\mu-\nu)]=1/2[3d-bc-\sqrt{(}\Delta)]$.

Thus $\alpha$, $\beta$, $\gamma$ and $\delta$ are solutions of
the three equations $b\alpha-3\beta+c\gamma-b\delta=0$, $-c\alpha+b\beta-3d\gamma+c\delta=0$
and $-\nu\alpha+(2c-b^{2})\beta+bd\gamma+\mu\delta=0$. Since $\alpha$,
$\beta$, $\gamma$ and $\delta$ are only determined to within a
constant multiple, we may add a fourth equation of our choice to the
system. Let $s_{1}$, $s_{2}$, $s_{3}$ and $s_{4}$ be chosen so
that adding the equation $s_{1}\alpha+s_{2}\beta+s_{3}\gamma+s_{4}\delta=1$
produces a system that has a unique solution for $\alpha$, $\beta$,
$\gamma$ and $\delta$. So we have the system 
\[
\left(\begin{array}{cccc}
s_{1} & s_{2} & s_{3} & s_{4}\\
b & -3 & c & -b\\
-c & b & -3d & c\\
-\nu & 2c-b^{2} & bd & \mu
\end{array}\right)\left(\begin{array}{c}
\alpha\\
\beta\\
\gamma\\
\delta
\end{array}\right)=\left(\begin{array}{c}
1\\
0\\
0\\
0
\end{array}\right)
\]
We solve this by Cramer's Rule, although we may neglect to divide
by the determinant of the original matrix since we only want our solution
to within a constant multiple. This gives us 
\[
\alpha=\left|\begin{array}{cccc}
1 & s_{2} & s_{3} & s_{4}\\
0 & -3 & c & -b\\
0 & b & -3d & c\\
0 & 2c-b^{2} & bd & \mu
\end{array}\right|=\left|\begin{array}{ccc}
-3 & c & -b\\
b & -3d & c\\
2c-b^{2} & bd & \mu
\end{array}\right|
\]
Continuing in this manner, we obtain 
\[
\beta=-\left|\begin{array}{ccc}
b & c & -b\\
-c & -3d & c\\
-\nu & bd & \mu
\end{array}\right|\mbox{ , }\gamma=\left|\begin{array}{ccc}
b & -3 & -b\\
-c & b & c\\
-\nu & 2c-b^{2} & \mu
\end{array}\right|\mbox{ and }\delta=-\left|\begin{array}{ccc}
b & -3 & c\\
-c & b & -3d\\
-\nu & 2c-b^{2} & bd
\end{array}\right|
\]


\section{Generalizations}

\label{S:generalizations}

It is natural to ask when roots of polynomials of other degrees can
have common tails. Nothing is lost by restricting our investigation
to irreducible polynomials with two or more real roots.

The question for quadratic polynomials was considered by J. A. Serret
in the 1800's, it appears as a problem in various editions of a textbook
he wrote (\cite{Serret2}, \cite{Serret1}). His solution says that
the two roots have common tails iff a quadratic Diophantine equation
is solvable.

While this is not a very satisfying answer, it may well be the best
possible. The situation is complicated by the fact that for a quadratic
polynomial, elements of the splitting field do not have unique representations
of the form $(\alpha r_{1}+\beta)/(\gamma r_{1}+\delta)$. Of course
Theorem \ref{ad theorem} still applies, and the roots have common
tails if there are $\alpha,\beta,\gamma,\delta$with $r_{2}=(\alpha r_{1}+\beta)/(\gamma r_{1}+\delta)$
and $\ad(\alpha,\beta,\gamma,\delta)=1$.

There are certainly examples of quadratic polynomials $p(x)$ with
roots that do not have common tails, one is $p(x)=14x^{2}+3x-7$.
Its roots have continued fractions $[0;1,\overline{1,1,1,4,2}]$ and
$[-1;5,\overline{1,1,1,2,4}]$. 

To determine whether polynomials of degrees above $3$ have roots
with common tails, we need to consider the Galois groups of the polynomials
and their actions on the roots. For cubic polynomials with splitting
fields of degree $3$, the Galois group is $\Z_{3}$ and its action
cyclically permutes the roots. Going to quartic polynomials with four
real roots and splitting fields of degree $4$, their are two possible
Galois groups. One is $\Z_{4}$, with an action that cyclically permutes
the four roots. The other possibility is the Klein $4$-group, $\Z_{2}\times\Z_{2}$.
Since the action must be transitive on the roots, we have that $\Z_{2}\times\Z_{2}$
must be the permutation group consisting of the identity and the three
permutations $(r_{1}r_{2})(r_{3}r_{4})$, $(r_{1}r_{3})(r_{2}r_{4})$
and $(r_{1}r_{4})(r_{2}r_{3})$.

It is easy to find polynomials of degree $4$ with Galois groups $\Z_{2}\times\Z_{2}$
where all roots have common tails. The simplest possibility is to
let $(r_{1}r_{2})(r_{3}r_{4})$ be the operation of negation, and
to let one of the other permutations be reciprocal. For example, consider
the polynomial $p(x)=x^{4}-4x^{2}+1$. It has four roots, $r_{1}=\sqrt{2+\sqrt{3}}$,
$r_{2}=-\sqrt{2+\sqrt{3}}$, $r_{3}=\sqrt{2-\sqrt{3}}$ and $r_{4}=-\sqrt{2-\sqrt{3}}$.
Here, $(r_{1}r_{2})(r_{3}r_{4})$ is negation. Since $r_{1}r_{3}=\sqrt{(2+\sqrt{3})(2-\sqrt{3})}=1$,
and so on, we also have that $(r_{1}r_{3})(r_{2}r_{4})$ is reciprocal.

On the other hand, no polynomial with Galois group $\Z_{4}$ has roots
with common tails. The problem is that there are no integer $2\times2$
matrices with determinant $\pm1$ that have order $4$ in the multiplicative
group $PGL(2,\Q)$. (We thank Edward Hanson for pointing us to the
literature on this.)

More precisely, we are looking at the possible finite orders in the
composition group of linear fractional transformations. If $f$ is
a linear fractional transformation of order $n$, we have that $M(f^{n})=(M(f))^{n}$
contains $I$, the $2\times2$ identity matrix, and that $I\notin(M(f))^{k}$
for $k<n$. We will modify the argument used in \cite{Hanson}, which
deals with the related problem of finding matrices of given orders
with minimum dimension over $\Q$. Since we are looking for roots
with common tails, we want $f$ where the matrix in standard form
has absolute value of its determinant equal to $1$.

Picking a matrix $A\in M(f)$ with determinant $\pm1$, we have that
$A^{n}$ is $\pm I$, but that $A^{k}$ is not $\pm I$ for $0<k<n$.
Let $m(x)$ be the minimum polynomial of $A$, the unique monic polynomial
of lowest degree in $\Q[x]$ that has $A$ as a root. Since $A$ satisfies
its characteristic polynomial, the degree of $m(x)$ is less than
or equal to the dimension of $A$, which is $2$. Since $A$ satisfies
$x^{2n}-I=0$, we have that $m(x)$ divides $x^{2n}-1$. Since $m(x)$
is irreducible, it must be a factor of $x^{2n}-1$, where these factors
are cyclotomic polynomials. The cyclotomic polynomials of degree $2$
or less are $c_{1}=x-1$, $c_{2}=x+1$, $c_{3}=x^{2}+x+1$, $c_{4}=x^{2}+1$
and $c_{6}=x^{2}-x+1$. Since $m(x)$ has degree $1$ or $2$, we
have five possibilities. If $m(x)$ is $x-1$ or $x+1$, then $f$
has order $1$. If $m(x)=x^{2}+x+1$, then $A^{3}=I$, while $A^{2}=-A-I\neq\pm I$
and $A\neq\pm I$, so $f$ has order $3$. If $m(x)=x^{2}+1$, then
$A^{2}$ is $-1$ and $f$ has order $1$ or $2$. If $m(x)=x^{2}-x+1$,
then $A$ is a root of $(x+1)(x^{2}-x+1)=x^{3}+1$, so $A^{3}=-I$
and $f$ has order $3$ or less. This gives us the following lemma.

\begin{lemma}\label{order lemma} Let $f$ be a linear fractional
transformation over $\Q$ with standard matrix with determinant $\pm1$.
If $f$ has finite order, then $f$ has order $1$, $2$ or $3$.
\end{lemma}

We note that the restriction that the determinant have absolute value
$1$ is necessary. In \cite{Dresden}, G. Dresden shows that $f(x)=(x-1)/(x+1)$
has order $4$ and that $g(x)=(2x-1)/(x+1)$ has order $6$ in $PGL(2,\Q)$.

Now suppose that there is an irreducible polynomial $p(x)$ over $\Q$
with Galois group $\Z_{4}$, and that the four distinct real roots
of $p(x)$ have common tails. Letting $f$ be a generator of the Galois
group $\Z_{4}$, we may call the roots $r_{1}$, $r_{2}=f(r_{1})$,
$r_{3}=f(r_{2})$ and $r_{4}=f(r_{3})$. Since $r_{1}$ and $r_{2}$
have common tails, we must have $r_{2}=(\alpha r_{1}+\beta)/(\gamma r_{1}+\delta)$
where $\ad(\alpha,\beta,\gamma,\delta)=1$. Applying $f$ to this
repeatedly, we get $r_{3}=(\alpha r_{2}+\beta)/(\gamma r_{2}+\delta)$,
$r_{4}=(\alpha r_{3}+\beta)/(\gamma r_{3}+\delta)$, and $r_{1}=(\alpha r_{4}+\beta)/(\gamma r_{4}+\delta)$.
Thus the linear fractional transformation $(\alpha z+\beta)/(\gamma z+\delta)$
has order $4$, contradicting Lemma \ref{order lemma}.

Extending this argument, we have the following theorem.

\begin{theorem}\label{2^k3^m theorem} Let $q(x)$ be an irreducible
polynomial of degree $n$ over $\Q$ with $n$ real roots $r_{1},r_{2},\dots r_{n}$.
If all these roots have common tails, $n$ must be of the form $2^{k}3^{m}$
for nonnegative integers $k$ and $m$. \end{theorem}

\begin{proof} Suppose $q(x)$ is as above, where $n$ is not of the
form $2^{k}3^{m}$. Then there is a prime $p$ dividing $n$ with
$p\geq5$. Let $G$ be the Galois group of $q(x)$. Since $q(x)$
is irreducible, $G$ acts transitively on $\{r_{1},r_{2},\dots r_{n}\}$.
Then by the Orbit-stabilizer Theorem, $|G|$ is $n$ times the order
of the stabilizer subgroup of any $r_{i}$, so $p$ divides $|G|$.
Cauchy's Theorem now implies that $G$ has an element $g$ of order
$p$. Letting $r_{i}$ and $r_{j}$ be distinct roots with $g(r_{i})=r_{j}$,
we write $r_{j}$ as a linear fractional transformation of $r_{i}$.
The order of this transformation must then be $p$, contradicting
Lemma \ref{order lemma}. \end{proof}

We do not know for which $n$ of the form $2^{k}3^{m}$ there are
irreducible polynomials over $\Q$ with $n$ real roots with common
tails. If m = 0, then the Galois group has $2^{k}$ many elements,
all of which have orders 1 or 2. This implies that the group is abelian,
and thus must be isomorphic to $Z_{2}^{k}$. Such a group would be
generated by $k$ commuting elements of order $2$. We do not believe
this is possible for $k>2$, since we can not find three distinct
matrices with integer entries, $A,B,C$ that meet all the requirements.
(We need $A^{2}=\pm I$, $B^{2}=\pm I$, $C^{2}=\pm I$, $AB=\pm BA$,
$AC=\pm CA$ and $BC=\pm CB$, plus some minor conditions.) Similarly,
we do not believe there are irreducible polynomials of degree $9$where
all the roots have common tails. For the Galois group would need to
be isomorphic to $Z_{3}\times Z_{3}$, and we have not found sufficiently
distinct integer matrices $A$and $B$with $A^{3}=\pm I$ , $B^{3}=\pm I$
, and $AB=\pm BA$. There are however examples with $k$ and $m$
both positive. Consider $n=6$. To avoid a $6$-cycle, the Galois
group $G$ must be isomorphic to $S_{3}$. If $g\in G$ has say $g(r_{1})=r_{2}$
where $r_{2}=(\alpha r_{1}+\beta)/(\gamma r_{1}+\delta)$ for $\alpha,\beta,\gamma,\delta\in\Z)$,
then we must have $g(r_{j})=(\alpha r_{j}+\beta)/(\gamma r_{j}+\delta)$
for all roots $r_{j}$ since $g$ fixes $\Q$. But $r_{j}$ can not
equal $(\alpha r_{j}+\beta)/(\gamma r_{j}+\delta)$, since $r_{j}$
is not the root of a quadratic. Thus every element of $G$ except
the identity must move all six roots.

For ease of notation, we consider $G\cong S_{3}$ to act on $\{1,2,3,4,5,6\}$,
rather than on the set of roots. So what actions are possible? Without
loss of generality, we let one element of order $3$ in $S_{3}$ be
$\rho=(123)(456)$, and let $\sigma$ be an element of order $2$.
We can not have $\sigma(1)=2$, since this would give us $\sigma(r_{1})=r_{2}=\rho(r_{1})$
and have both $\sigma$ and $\rho$ represented by the same linear
fractional transformation on the roots. Considering $\sigma$ and
$\rho^{-1}$, we see that $\sigma(1)$ can also not be $3$. So without
loss of generality, $\sigma$ contains the cycle $(14)$. By similar
arguments, $\sigma$ must take $2$ into $\{5,6\}$, and take $3$
to whichever of $5$ or $6$ is left. But we can not have $\sigma=(14)(25)(36)$,
for then $\sigma$ would commute with $\rho$, which can not happen
in $S_{3}$. So $\sigma=(14)(26)(35)$. Now $\sigma$ and $\rho$
generate $S_{3}$, and determine an action on $\{r_{1},r_{2},r_{3},r_{4},r_{5},r_{6}\}$
that is unique up to renaming the roots.

To implement this action on the roots, we first pick a combination
of fractional linear transformations for $\rho$ and $\sigma$. We
know that $\rho$ must have order $3$, $\sigma$ must have order
$2$, and they must satisfy $\sigma\circ\rho=\rho^{-1}\circ\sigma$.
(These equations technically only need to hold at the $6$ roots,
but those are enough values to insure that the functions are equal.)
We write $\rho$ as $(\alpha x+\beta)/(\gamma x+\delta)$, where $\ad(\alpha,\beta,\gamma,\delta)=1$.
For $\rho$ to have order $3$, we must have $\delta=-1-\alpha$ and
$\beta\gamma=-(1+\alpha+\alpha^{2})$ , as in the discussion following
Theorem $\ref{same tail criterion}$. 

Now we need to find $\sigma$, and try one of the simpler choices
letting $\sigma(x)=1/x$. Then $\sigma\circ\rho=\rho^{-1}\circ\sigma$
becomes $(\gamma x+\delta)/(\alpha x+\beta)=\pm(\delta/x-\beta)/(-\gamma/x+\alpha)=\pm(-\beta x+\delta)/(\alpha x-\gamma)$,
which yields $\gamma=\mp\beta$, $\delta=\pm\delta$ and $\beta=-\gamma$.
Letting $\alpha\delta-\beta\gamma=-1$ gives $\gamma=\beta$, $\delta=-\delta$
and $\beta=-\gamma$, implying $\beta=\gamma=\delta=0$ which will
not work. So we let $\alpha\delta-\beta\gamma=1$, giving $\gamma=-\beta$,
$\delta=\delta$ and $\beta=-\gamma$. This is compatible with our
other conditions in a few cases; letting $\alpha=0$, $\delta=-1$,
$\beta=-1$ and $\gamma=1$ works.

So we take $\rho(x)=-1/(x-1)$ and $\sigma(x)=1/x$ on the roots.
Now we proceed to hunt for a $6$th degree polynomial $p(x)$ with
coefficients in $\Z$ that allows our $\rho$ and $\sigma$ in its
Galois group. We could use matrices and eigenvectors to find this
as after Theorem $\ref{same tail criterion}$, but will instead deal
with the polynomials directly. The presence of $\sigma$ means that
whenever $r$ is a root of $p(x)$, so is $1/r$. On the other hand,
we can let $p(x)$ be $ax^{6}+bx^{5}+cx^{4}+dx^{3}+ex^{2}+fx+g$,
and observe that for non-zero $x$, $a+b(1/x)+c(1/x)^{2}+d(1/x)^{3}+e(1/x)^{4}+f(1/x)^{5}+g(1/x)^{6}=0$
iff $ax^{6}+bx^{5}+cx^{4}+dx^{3}+ex^{2}+fx+g=0$. That is, $1/r$
is a root of $gx^{6}+fx^{5}+ex^{4}+dx^{3}+cx^{2}+bx+a$ when $r$
is a root of $p(x)$. But $1/r$ is a root of $p(x)$, which is its
minimum polynomial. Thus $gx^{6}+fx^{5}+ex^{4}+dx^{3}+cx^{2}+bx+a$
is a multiple of $p(x)$. Since the greatest common divisors of both
sets of coefficients are equal, the polynomial is $\pm p(x)$. It
works to have it be $p(x)$, so we equate coefficients, and get $a=g$,
$b=f$ and $c=e$.

Similarly, we factor $\rho(x)=-1/(x-1)$ into subtracting $1$ from
$x$, taking the reciprocal, and negating the result. Each step corresponds
to an operation on the polynomial. We have that when $r$ is a root
of $p(x)$, $r-1$ is a root of $p(x+1)=ax^{6}+(6a+b)x^{5}+(15a+5b+c)x^{4}+(20a+10b+4c+d)x^{3}+(15a+10b+6c+3d+e)x^{2}+(6a+5b+4c+3d+2e+f)x+(a+b+c+d+e+f+g)$.
Reciprocals of the roots correspond to reversing the coefficients,
giving the polynomial $(a+b+c+d+e+f+g)x^{6}+(6a+5b+4c+3d+2e+f)x^{5}+(15a+10b+6c+3d+e)x^{4}+(20a+10b+4c+d)x^{3}+(15a+5b+c)x^{2}+(6a+b)x+a$
Negating the roots corresponds to negating coefficients of odd powers
of $x$, so we get that when $r$ is a root of $p(x)$, $\rho(r)$
is a root of the polynomial $(a+b+c+d+e+f+g)x^{6}-(6a+5b+4c+3d+2e+f)x^{5}+(15a+10b+6c+3d+e)x^{4}-(20a+10b+4c+d)x^{3}+(15a+5b+c)x^{2}-(6a+b)x+a$.
But $\rho(r)$ is another root of $p(x)$, which is its minimum polynomial.
So this polynomial is a multiple of $p(x)$. Since $GCD(a+b+c+d+e+f+g,6a+5b+4c+3d+2e+f,15a+10b+6c+3d+e,20a+10b+4c+d,15a+5b+c,6a+b,a)=GCD(a,b,c,d,e,f)$,
the polynomial is $\pm1$ times $p(x)$. It works to have it equal
$p(x)$, so we equate coefficients and get $a+b+c+d+e+f+g=a$, $6a+5b+4c+3d+2e+f=-b$,$15a+10b+6c+3d+e=c$,
$20a+10b+4c+d=-d$, $15a+5b+c=e$, $6a+b=-f$, and $a=g$. Substituting
in $g=a$, $f=b$, $e=c$, and simplifying, the system reduces to
the equations $b=f=-3a$, $d=5a-2c$, $e=c$ and $g=a$, leaving us
free to choose $a$ and $c$. One choice that gives an irreducible
polynomial with six real roots is $a=1$ and $c=-4$, giving the polynomial
$p(x)=x^{6}-3x^{5}-4x^{4}+13x^{3}-4x^{2}-3x+1$. Mathematica confirms
that the roots have common tails and that the Galois group is generated
by our $\rho$ and $\sigma$.

This approach may also work for higher degree polynomials, although
we have not investigated further than the following. To get an irreducible
$12$th-degree polynomial with common tails, there is essentially
only one possibile action of its Galois group on the roots. The Galois
group would be a \textbf{$12$} element group. This group would have
to be isomorphic to $A_{4}$, as the other four $12$ element groups
have elements with order greater than $3$. (See \cite{Coxeter},
for instance.) Since the polynomial is irreducible, the Galois group
would act transitively on the roots. Up to renaming the roots, there
is one possible transitive action of $A_{4}$ on them, which is isomorphic
to the action given by Cayley's Theorem. (Identifying the roots with
the numbers $\ensuremath{1,2,3,\dots12}$, use transitivity and renaming
of roots to get elements $\rho=(123)(456)(789)(10\;11\;12)$ and $\sigma=(14)(28)(3\;12)(5\;11)(69)(7\;10)$
in $A_{4}$. Then \textbf{$\rho$} and \textbf{$\sigma$} generate
$A_{4}$.)


\end{document}